\documentclass[reqno]{amsart}
\usepackage{amssymb,graphicx}
\usepackage{tikz}
\usetikzlibrary{shapes,arrows}
\usetikzlibrary{shapes.geometric,arrows,chains}
\usepackage{mathrsfs}
\usepackage[hidelinks]{hyperref}
\hypersetup{colorlinks=true,linkcolor=blue}
\usepackage{setspace}
\usepackage[fontsize=11.3pt]{fontsize}
\usepackage{bm}

\usepackage{geometry}
\allowdisplaybreaks[4]
\theoremstyle{plain}
\newtheorem{thm}{Theorem}[section]
\newtheorem{lem}[thm]{Lemma}
\newtheorem{cor}[thm]{Corollary}

\newtheorem{prop}[thm]{Proposition}
\newtheorem{rmk}[thm]{Remark}

\numberwithin{equation}{section}

\def\b{\beta}
\def\d{\Delta}
\def\e{\epsilon}

\def\l{\mathscr{L}}
\def\M{\mathbf{M^n}}
\def\O{\Omega}

\def\p2{\mathscr{P}_2(X)}

\def\bt{\bar{t}}
\def\bx{\bar{x}}

\def\x{x^{\star}}

\makeatletter 
\renewcommand{\section}{\@startsection{section}{1}{0mm}
	{-\baselineskip}{0.5\baselineskip}{\bf\leftline}}
\makeatother
\makeatletter
\@namedef{subjclassname@2020}{\textup{2020} Mathematics Subject Classification}
\makeatother

\begin{document}
	
	\title[Liouville theorems and Harnack inequalities for Allen-Cahn type equation]{Liouville theorems and Harnack inequalities\\ for Allen-Cahn type equation}
	
	\author[Zhihao Lu]{Zhihao Lu}
	\address[Zhihao Lu]{School of Mathematics and Statistics, Jiangsu Normal University, Xuzhou 221116, P. R. China}
	\email{\href{mailto: Zhihao Lu <lzh139@mail.ustc.edu.cn>}{lzh139@mail.ustc.edu.cn}}

	\begin{abstract}
		We first give a logarithmic gradient estimate for the local positive solutions of Allen-Cahn equation on the complete Riemannian manifolds with Ricci curvature bounded below. As its natural corallary, Harnack inequality and  a  Liouville theorem for classical positive solutions are obtained. Later, we consider similar estimate under integral curvature condition and generalize previous results to a class nonlinear equations which contain some classical elliptic equations such as Lane-Emden equation, static Whitehead-Newell equation and static Fisher-KPP equation. Last, we briefly generalize them to equation with gradient item under Bakry-\'{E}mery curvature condition.
	\end{abstract}
	
	
	
	\thanks{School of Mathematics and Statistics, Jiangsu Normal University, Xuzhou 221116, P. R. China}
	
	
	\maketitle
	\section{\textbf{Introduction}}
	
	Establishing Harnack inequality of positive solutions and searching Liouville property are interesting topics in the study of elliptic equations. In the paper, we establish them for the following Allen-Cahn equation and its suitable generalizations. Concretely speaking, we first consider the classical positive solutions of

	\begin{equation}\label{AC}
		\Delta u+u-u^3=0
	\end{equation}
	on a domain in the Riemannian manifolds  which possess Ricci curvature bounded below. For simplicity, we assume that the domain is geodesic ball and denote by $B(x_0,R)$ the open geodesic ball with center $x_0$ and radius $R$. Now, we state our results for equation \eqref{AC}. First, we get  a logarithmic gradient estimate for classical positive solutions of \eqref{AC}.
	
	\begin{thm}\label{main1}
		Let $(\M,g)$ be an n-dimensional complete Riemannian manifold with $Ric\ge-Kg$, where $K\ge 0$. If $u$ is a positive solution of \eqref{AC} on $B(x_0,2R)$, then 
		\begin{equation}\label{strong1}
			\sup_{B(x_0,R)}\frac{|\nabla u|^2}{u^2}\le \frac{\left(K-\frac{2}{n}(1+\b)\right)^{+}}{\frac{1}{n}(1+\b)^2-\b^2-\frac{\e}{2}\b^2}+\frac{C(n,\b,\epsilon)}{R^2}+\frac{C(n,\b,\e)\sqrt{K}}{R},
		\end{equation}	
		where $\b\in(0,\frac{1}{\sqrt{n}-1})\cap(0,n-1]$ and $\e\in(0,\frac{2}{n}(1+\frac{1}{\b})^2-2)$ are arbitrary.

	\end{thm}
	From Theorem \ref{main1}, if $u$ is a positive solution on $\M$ with $Ric\ge-Kg$, then we have 
	\begin{equation}\label{ank}
		\sup_{\mathbf{M^{n}}}|\nabla\ln u|^2\le\inf\limits_{\b\in(0,\frac{1}{\sqrt{n}-1})\cap(0,n-1]}\frac{\left(K-\frac{2}{n}(1+\b)\right)^{+}}{\frac{1}{n}(1+\b)^2-\b^2}=:A(n,K).
	\end{equation}
	Here, we care the constant in the RHS of \eqref{strong1} because it can provide a Liouville theorem on negative curved space such as hyperbolic space. If we choose a fixed $\b=\b(n)$ and $\e=\e(n)$, then we deduce a weak form of \eqref{strong1}:
	\begin{equation}
		\sup_{B(x_0,R)}\frac{|\nabla u|^2}{u^2} \le C(n)\left(K+\frac{1}{R^2}\right).
	\end{equation}
	
	A direct corollary of Theorem \eqref{main1} is the following Harnack inequality. 
	\begin{cor}
		Let $(\M,g)$ be an n-dimensional complete Riemannian manifold with $Ric\ge-Kg$, where $K\ge 0$. If $u$ is a positive solution of \eqref{AC} on $B(x_0,2R)$, then
		\begin{equation}\label{HKAC}
			\sup\limits_{B(x_0,R)} u\le e^{C\left(\sqrt{K}R+1\right)}\cdot\inf\limits_{B(x_0,R)}u,
		\end{equation}
		where $C=C(n)$. Especially, if the solution on $\M$, then $\sup_{B(x_0,R)}u\le e^{\sqrt{A(n,K)}R}\inf_{B(x_0,R)}u$, where $A(n,K)$ is defined by \eqref{ank}.
	\end{cor}
	Another  corollary of Theorem \ref{main1} is the following Liouville theorem. For simple description, for any $n\in(1,\infty)$, we define universal constant
	\begin{equation}\label{N}
		N(n)=\sup\limits_{\b\in(0,\frac{1}{\sqrt{n}-1})\cap(0,n-1]}\frac{2}{n}(1+\b).
	\end{equation}

	\begin{thm}\label{main2}
		Let $(\M,g)$ be an n-dimensional complete Riemannian manifold with $Ric\ge-Kg$ and $K\le N(n)$, where $N$ is defined by \eqref{N}. If $u$ is a positive classical solution of \eqref{AC} on $\M$, then $u\equiv1$.
		
	\end{thm}
	\begin{rmk}
		\rm	For $n=1$, we do not need Ricci condition, and for $n=2$, $N(2)=2$. If $n\ge 3$, $N(n)=\frac{2}{\sqrt{n}(\sqrt{n}-1)}$. Hence Theorem \ref{main2} implies that the global positive solutions of \eqref{AC} on hyperbolic plane $\mathbb{H}^2$ is constant 1. Theorem \ref{main2} seems to be the first Liouville theorem of elliptic equation on negative curved manifolds.
	\end{rmk}

	If we consider equation \eqref{AC} on closed manifolds, via the maximal principle, we have
	\begin{thm}\label{main3}
		Let $(\M,g)$ be an n-dimensional closed manifold. If $u$ is a positive solution of \eqref{AC} on $\M$, then $u\equiv1$.	
	\end{thm}
	\begin{rmk}
		\rm	Actually, any entire positive solution of equation \eqref{AC} on complete Riemannian manifolds with Ricci bounded below must be less than or equal to 1, see \cite[Theorem 1.1, Corollary 1.2]{LU}.
	\end{rmk}
	\begin{rmk}
		\rm	Before our consideration, Hou \cite[Theorem 1.1,Theorem 1.2]{H} obtained similar results as Theorem \ref{main1} and Theorem \ref{main2} with additional condition $0<u\le 1$ and nonnegative Ricci curvature and their proof can not remove these conditions. 
	\end{rmk}
	
	The rest of this paper is organized as follows. We prove Theorem \ref{main1}, Theorem \ref{main2} and Theorem \ref{main3} in Section \ref{S2}. In Section \ref{S3}, we give up the condition "Riemannian manifold with Ricci curvature bounded below" and consider same estimate under "$L^p$ Ricci curvature pertubation" condition which both contain non-negative Ricci curvature case. Later, in Section \ref{S4}, we generalize same estimates in Section \ref{S2} and Section \ref{S3} to more nonlinear elliptic equations which contains Lane-Emden equation, static Whitehead-Newell equation and static  Fisher-KPP equation. Meanwhile, we get Liouville theorems for them on complete Riemannian manifolds with non-negative Ricci curvature. Last, we slightly  generalize results in Section \ref{S4} to equations with gradient item under Bakry-\'{E}mery curvature condition.

	\section{\textbf{Logarithmic gradient estimate for Allen-Cahn equation}}\label{S2}
	In this section, we give a proof of Theorem \ref{main1} via Bernstein method.
	
	Let $u$ be a positive smooth solution of \eqref{AC} on a domain $\O\subset\M$. First, for $\b\neq 0$, we set
	\begin{equation}\label{tran}
		w=u^{-\beta},
	\end{equation}
	and auxiliary function
	\begin{equation}\label{F}
		F=\frac{|\nabla w|^2}{w^2}.
	\end{equation}

	Then we have the following technical proposition.
	\begin{prop}\label{k1}
		Let $u$ be a positive smooth solution of \eqref{AC} on a domain $\O\subset\M$ and $F$ be defined by \eqref{F} for some $\b\neq 0$. Then we have
		\begin{eqnarray}
			\Delta F&=&2w^{-2}\left|\nabla^2w-\frac{\Delta w}{n} g\right|^2+2w^{-2}Ric(\nabla w,\nabla w)+2\left(\frac{1}{\b}-1\right)\left\langle\nabla F,\nabla \ln w\right\rangle\nonumber\\
			&&+\left(\frac{4}{n}(1+\b)+4\left(1-\frac{1+\b}{n}\right)u^2\right)\frac{|\nabla w|^2}{w^2}\nonumber\\
			&&+\left(\frac{2}{n}\left(1+\frac{1}{\b}\right)^2-2\right)\frac{|\nabla w|^4}{w^4}+\frac{2\b^2}{n}\left(1-u^2\right)^2.
		\end{eqnarray}

	\end{prop}
	\begin{proof}[\textbf{Proof}]
		By equation \eqref{AC} and \eqref{tran}, we have
		\begin{equation}\label{weq}
			\d w=\left(1+\frac{1}{\b}\right)\frac{|\nabla w|^2}{w}+\b w(1-u^2).
		\end{equation}
		
		Chain rule gives
		\begin{equation}\label{p1}
			\d\left(\frac{|\nabla w|^2}{w^2}\right)=\d\left(w^{-2}\right)|\nabla w|^2+\d|\nabla w|^2\cdot w^{-2}+2\langle\nabla (w^{-2}),\nabla|\nabla w|^2\rangle.
		\end{equation}
		
		By equation \eqref{weq} and chain rule again, we deduce 
		\begin{eqnarray}\label{p2}
			\d(w^{-2})|\nabla w|^2=\left(4-\frac{2}{\b}\right)\frac{|\nabla w|^4}{w^4}-2\b(1-u^2)\frac{|\nabla w|^2}{w^2}
		\end{eqnarray}
		
		and
		\begin{equation}\label{p3}
			2\langle\nabla (w^{-2}),\nabla|\nabla w|^2\rangle=-8w^{-3}\nabla^2w(\nabla w,\nabla w).
		\end{equation}
		
		Using Bochner formula and equation \eqref{weq} again, one derives
		\begin{eqnarray}\label{p4}
			\d|\nabla w|^2\cdot w^{-2}&=&2w^{-2}\left(|\nabla^2 w|^2+Ric(\nabla w,\nabla w)+\langle\nabla\d w,\nabla w\rangle\right)\nonumber\\
			&=&2w^{-2}\left(|\nabla^2 w|^2+Ric(\nabla w,\nabla w)\right)+2\left(\b(1-u^2)+2u^2\right)\frac{|\nabla w|^2}{w^2}\nonumber\\
			&&-2\left(1+\frac{1}{\b}\right)\frac{|\nabla w|^4}{w^4}+4\left(1+\frac{1}{\b}\right)w^{-3}\nabla^2w(\nabla w,\nabla w).
		\end{eqnarray}
		
		Plugging \eqref{p2}-\eqref{p4} into \eqref{p1}, we see
		
		\begin{eqnarray}\label{p5}
			\d\left(\frac{|\nabla w|^2}{w^2}\right)&=&2w^{-2}\left(|\nabla^2 w|^2+Ric(\nabla w,\nabla w)\right)\nonumber\\
			&&+\left(2-\frac{4}{\b}\right)\frac{|\nabla w|^4}{w^4}+4u^2\frac{|\nabla w|^2}{w^2}+4\left(\frac{1}{\b}-1\right)w^{-3}\nabla^2(\nabla w,\nabla w).
		\end{eqnarray}
		
		Notice that
		\begin{equation}\label{p6}
			4\left(\frac{1}{\b}-1\right)w^{-3}\nabla^2 w(\nabla w,\nabla w)=2\left(\frac{1}{\b}-1\right)\langle\nabla F,\nabla\ln w\rangle+4\left(\frac{1}{\b}-1\right)\frac{|\nabla w|^4}{w^4}
		\end{equation}
		
		and	
		\begin{equation}\label{p7}
			|\nabla^2 w|^2=\left|\nabla^2 w-\frac{\d w}{n}g\right|^2+\frac{1}{n}\left(\left(1+\frac{1}{\b}\right)\frac{|\nabla w|^2}{w}+\b w(1-u^2)\right)^2.
		\end{equation}
		
		Now, substituting \eqref{p6} and \eqref{p7} into \eqref{p5} yields our desired result.

	\end{proof}
	Via the maximal principle, we directly prove the Theorem \ref{main3}.
	\begin{proof}[\textbf{Proof of Theorem \ref{main3}}]
	Assume that  the positive solution $u$ gets its maximal value at $\x$, then $\Delta u(\x)\le 0$, which implies $u(\x)\le 1$. So $$\Delta u=u(u^2-1)\le 0.$$
		By the maximal principle on closed manifolds,  $u$ must be constant and so $u\equiv 1$.  
	\end{proof}
	
	
	


	\begin{proof}[\textbf{Proof of Theorem \ref{main1}}]
		Set $\O=B(x_0,2R)$  and $F$ is same as in the Proposition \ref{k1}. First, by Laplacian comparison theorem, we can choose smooth cut-off function $\Phi$ such that \\
		(i) $\Phi(x)=\phi(d(x_0,x))$, here $\phi$ is a nonincreasing function on $[0,\infty)$ and
		\begin{eqnarray}
			\Phi(x)=
			\begin{cases}
				1\qquad\qquad\qquad\,\text{if}\qquad x\in B(x_0,R)\nonumber\\
				0\qquad\qquad\qquad\,\text{if}\qquad x\in B(x_0,2R)\setminus B(x_0,\frac{3}{2}R).
			\end{cases}
		\end{eqnarray}
		(ii) For any fixed $l\in(0,1)$,
		\begin{equation}
			\frac{|\nabla \Phi|}{\Phi^{l}}\le\frac{C(l)}{R},\nonumber
		\end{equation}
		(iii)
		\begin{equation}
			\Delta \Phi\ge-\frac{C(n)\sqrt{K}}{R}-\frac{C(n)}{R^2}\nonumber
		\end{equation}
		on $B(x_0,2R)$ and outside of cut locus of $x_0$.
		
		Case 1: $n=1$. We choose $\b=1$, then Proposition \ref{k1} gives
		\begin{eqnarray}\label{31}
			\Delta F&\ge&4\left(1-u^2\right)\frac{|\nabla w|^2}{w^2}+6\frac{|\nabla w|^4}{w^4}+2\left(1-u^2\right)^2
			\ge4F^2.
		\end{eqnarray}		
		Choosing $\Phi$ as above, without loss of generality, at the maximal point $x$ of $\Phi F$, we assume $\Phi F(x)>0$ and the point $x$ is outside of cut locus of $x_0$, or we can use Calabi's argument as in Li-Yau \cite{LY}, then we have
		\begin{eqnarray}\label{32}
			0&\ge& \d (\Phi F)(x)=\d \Phi\cdot F+\Phi\d F-2\frac{|\nabla \Phi|^2}{\Phi}F.
		\end{eqnarray}
		Using the property of $\Phi$, substituting \eqref{31} into  \eqref{32} yields	
		\begin{eqnarray}\label{33}
			0&\ge& 4(\Phi F)^2-\frac{C}{R^2}\Phi F,
		\end{eqnarray}	
		so	
		\begin{equation}
			\sup_{B(x_0,R)}F\le \Phi F(x)\le\frac{C}{R^2}.
		\end{equation}	
		
		Case 2: $n> 1$.	By Proposition \ref{k1}, if we choose any $\b\in(0,\frac{1}{\sqrt{n}-1})$ and $\b\le n-1$, then  we have
		\begin{eqnarray}\label{34}
			\Delta F&\ge&2\left(\frac{1}{\b}-1\right)\left\langle\nabla F,\nabla \ln w\right\rangle+\left(\frac{2}{n}\left(1+\frac{1}{\b}\right)^2-2\right)F^2-2\left(K-\frac{2}{n}(1+\b)\right)F.
		\end{eqnarray}	
		Choosing $\Phi$ as before, 	by same argument as in Case 1, combining \eqref{34} and \eqref{32}, we conclude that at the maximal point $x$ of $\Phi F$, we have
		\begin{eqnarray}
			0&\ge& \d\Phi \cdot(\Phi F)-2\frac{|\nabla \Phi|^2}{\Phi}(\Phi F)-2\left(K-\frac{2}{n}(1+\b)\right)\Phi^2F\nonumber\\
			&&+\left(\frac{2}{n}\left(1+\frac{1}{\b}\right)^2-2\right)(\Phi F)^2-2\left(\frac{1}{\b}-1\right)\left\langle\nabla \Phi,\nabla \ln w\right\rangle(\Phi F)\nonumber\\
			&\ge&\d\Phi \cdot(\Phi F)-2\frac{|\nabla \Phi|^2}{\Phi}(\Phi F)-2\left(K-\frac{2}{n}(1+\b)\right)\Phi^2F\nonumber\\
			&&+\left(\frac{2}{n}\left(1+\frac{1}{\b}\right)^2-2-\e\right)(\Phi F)^2-C(n,\epsilon)\frac{|\nabla \Phi|^2}{\Phi}(\Phi F),
		\end{eqnarray}	
		where we use Cauchy-Schwarz inequality and basic inequality in the last inequality and choose any $\epsilon\in(0,\frac{2}{n}(1+\frac{1}{\b})^2-2)$.  Hence	
		\begin{eqnarray}
			\sup_{B(x_0,R)}\frac{|\nabla u|^2}{u^2}&=&	\b^{-2}\sup_{B(x_0,R)}F\le \b^{-2}\Phi F(x)\nonumber\\
			&\le&  \frac{\left(K-\frac{2}{n}(1+\b)\right)^{+}}{\frac{1}{n}(1+\b)^2-\b^2-\frac{\e}{2}\b^2}+\frac{C(n,\b,\epsilon)}{R^2}+\frac{C(n,\b,\e)\sqrt{K}}{R}.
		\end{eqnarray}
		Then we finish 	the proof.

	\end{proof}

	\section{\textbf{Logarithmic gradient estimate under Ricci curvature pertubation}}\label{S3}
	
	In this section, we consider corresponding estimate in Section \ref{S2} under Ricci curvature pertubation. Concretely, we introduce the following quantity as in Dai-Wei-Zhang \cite{DWZ}. Let $(\M,g)$ be an $n$-dimensional Riemannian manifold, for $p>\frac{n}{2}$, we define
	\begin{equation}
		k(x,p,R,K)=R^2\left(\frac{1}{|B(x,R)|}\int_{B(x,R)}(Ric_K^-(y))^pdy\right)^{\frac{1}{p}}
	\end{equation}
	and
	\begin{equation}
		k(p,R,K)=\sup_{x\in\mathbf{M^{n}}}k(x,p,R,K),
	\end{equation}
	where we define $\rho(x)$ be the smallest eigenvalue for Ricci tensor Ric: $T_xM\to T_xM$ and $Ric_K^-(x)=((n-1)K-\rho(x))^{+}=\max((n-1)K-\rho(x),0)$. If $K=0$, we denote it by $Ric^-$. In this special case, $k(x,p,R,K)$ and $k(p,R,K)$ are scale invariant and we denote them by $k(x,p,R)$ and $k(p,R)$ respectively.  
	
	On the complete manifolds with non-negative Ricci curvature, from Theorem \ref{main1}, we can see that any classical positive solution of \eqref{AC} on $B(x_0,2R)\subset\M$ satisfies  
	\begin{equation}
		\sup_{B(x_0,R)}\frac{|\nabla u|^2}{u^2}\le \frac{C(n)}{R^2}.
	\end{equation}	
	
	A suitable generalization of above estimate is the following theorem.
	\begin{thm}\label{3001}
		Let $(\M,g)$ be an n-dimensional complete Riemannian manifold. If $u$ is a positive solution of \eqref{AC} on $B(x_0,2R)$, then for 
		any $p>\frac{n}{2}$, there exist constants $k=k(n,p)>0$ and $C=C(n,p)$ such that if $k(p,R)\le k$, we have
		\begin{equation}
			\sup_{B(x_0,R)}\frac{|\nabla u|^2}{u^2}\le \frac{C}{R^2}.
		\end{equation}	
		
	\end{thm}
	
	Before we prove Theorem \ref{3001}, we briefly mention  main strategy in the proof. The key obversation is that one can remove the Ricci item by an evolutionary equation  which links an undetermined positive function that also depends on time. For this, we define $\l=\d-\partial_t$ and compute the corresponding evolutionary equation of auxiliary function. The strategy was first used for heat equation by Zhang-Zhu \cite{ZZ1,ZZ2}. For Allen-Cahn equation, we have the following lemma.
	
	\begin{lem}\label{301}
		Let $u$ be a positive smooth solution of \eqref{AC} on a domain $\O\subset\M$ and $F$ be defined by \eqref{F}. For an undetermined positive function $J(x,t)\in C^{2,1}(\O\times[0,\infty))$, there exist $\delta=\delta(n)\in(0,1)$ and $\b=\b(n)\in(0,1)$ such that
		\begin{equation}
			\l(JF)\ge\left(\l J-5\frac{|\nabla J|^2}{J\delta}-2Ric^{-}\right)F+2\left(\frac{1}{\b}-1\right)\left\langle\nabla (JF),\nabla\ln w\right\rangle +2JF^2.
		\end{equation}

	\end{lem}
	\begin{proof}[\textbf{Proof}]
		By chain rule, \eqref{p5} and \eqref{p6}, we have
		\begin{eqnarray}\label{311}
			\l (JF)&=&\l J\cdot F+2\left\langle\nabla J,\nabla F\right\rangle+
			J\left(2w^{-2}|\nabla^2 w|^2+2w^{-2}Ric(\nabla w,\nabla w)-2F^2+4u^2F\right)\nonumber\\
			&&+2\left(\frac{1}{\b}-1\right)\left\langle\nabla (JF),\nabla\ln w\right\rangle-2F\left(\frac{1}{\b}-1\right)\left\langle\nabla J,\nabla\ln w\right\rangle.
		\end{eqnarray}
		
		First, for any $\delta>0$, Cauchy-Schwarz inequality and basic inequality give
		\begin{eqnarray}\label{312}
			2\left\langle\nabla J,\nabla F\right\rangle&=&4w^{-2}\nabla^2w(\nabla J,\nabla w)-4w^{-3}\left\langle\nabla J,\nabla w\right\rangle|\nabla w|^2\nonumber\\
			&\ge&-\left(2J\delta\frac{|\nabla^2 w|^2}{w^2}+\frac{2|\nabla J|^2}{J\delta}\cdot\frac{|\nabla w|^2}{w^2}\right)-\left(2J\delta\frac{|\nabla w|^2}{w^2}+\frac{2|\nabla J|^2}{J\delta}\right)\frac{|\nabla w|^2}{w^2},\nonumber\\
			&&	
		\end{eqnarray}
		and 
		\begin{eqnarray}\label{313}
			-2F\left(\frac{1}{\b}-1\right)\left\langle\nabla J,\nabla\ln w\right\rangle
			&\ge&-J\delta\left(\frac{1}{\b}-1\right)^2\frac{|\nabla w|^4}{w^4}-\frac{|\nabla J|^2}{J\delta}\cdot\frac{|\nabla w|^2}{w^2}.
		\end{eqnarray}
		
		Substituting \eqref{312} and \eqref{313} into \eqref{311} and using \eqref{p7}, we obtain
		\begin{eqnarray}
			\l(JF)&\ge&\left(\l J-5\frac{|\nabla J|^2}{J\delta}-2Ric^{-}\right)F+2\left(\frac{1}{\b}-1\right)\left\langle\nabla (JF),\nabla\ln w\right\rangle\nonumber\\
			&&+\left(\frac{2(1-\delta)}{n}\left(1+\frac{1}{\b}\right)^2-2-2\delta-\delta\left(\frac{1}{\b}-1\right)^2\right)JF^2\nonumber\\
			&&+\left(\frac{4(1-\delta)}{n}(1+\b)(1-u^2)+4u^2\right)JF.
		\end{eqnarray}
		Then we complete the proof by choosing $\delta=\frac{1}{5n}$ and $\b$ sufficiently small and positive.

	\end{proof}
	Now, we define 
	$\mathscr{F}(\delta,J)=\l J-5\frac{|\nabla J|^2}{J\delta}-2Ric^{-}$,
	and the following technical lemma help us remove the item $\mathscr{F}(\delta,J)$ in Lemma \ref{301}. 
	
	\begin{lem}\label{C3}\cite[Claim]{ZZ2}\cite[Lemma 4.2]{LU0}
		Let $(\mathbf{M^{n}},g)$ be a Riemannian manifold. Then for any $p>\frac{n}{2}$, there exists a $k=k(n,p)$ such that when $k(p,R)\le k$,  the following problem has unique positive solution on $B(x_0,R)\times[0,\infty)$ for any $\delta\in(0,\frac{1}{5})$,
		\begin{eqnarray}\label{a1}
			\begin{cases}
				\mathscr{F}(\delta,J)=0\qquad\,\, \text{on} \quad B(x_0,R)\times(0,\infty)\\
				J(\cdot,0)=1\quad\qquad \text{on} \quad B(x_0,R)\\
				J(\cdot,t)=1 \quad\qquad \,\text{on} \quad \partial B(x_0,R).\\
			\end{cases}
		\end{eqnarray}
		The solution of \eqref{a1} has the following decay estimate:
		\begin{equation}
			J_R(t)\le J(x,t)\le 1,
		\end{equation}
		where
		\begin{eqnarray}\label{jr}
			J_R(t)=2^{\frac{-1}{5\delta^{-1}-1}} \cdot  \exp\left\{-Ck\Big(1+[C(5\delta^{-1}-1)k]^{\frac{n}{2p-n}}\Big)R^{-2} t\right\}
		\end{eqnarray}
		and $C=C(n,p)$. 	
	\end{lem}
	Before we start to prove Theorem \ref{3001}, we need a cut-off function under integral curvature condition from Dai-Wei-Zhang \cite[Lemma 5.3]{DWZ}, here we use its scaling version.
	\begin{lem}\label{ecf}
		Let $(\mathbf{M^{n}},g)$ be a Riemannian manifold. Then for 
		any $p>\frac{n}{2}$, there exists constants $k=k(n,p)$ and $C=C(n,p)$ such that if $k(p,R)\le k$, then for any geodesic ball $B(x_0,2R)$, there exists $\phi\in C_0^{\infty}(B(x_0,2R))$ satisfying $0\le\phi\le 1$, $\phi=1$ in $B(x,R)$, and $|\nabla\phi|^2+|\Delta\phi|\le\frac{C}{R^2}$.  	
	\end{lem}
	
	\begin{proof}[\textbf{Proof of Theorem \ref{3001}}]
		First, we take $\O=B(x_0,2R)$ in Lemma \ref{301} and set the following auxiliary function	
		\begin{equation}
			G(x,t)=t\Phi(x)J(x,t)F(x)\qquad\text{on $B(x_0,2R)\times[0,\infty)$,}
		\end{equation}
		where $\Phi\in C_0^{\infty}(B(x_0,2R))$ is an undetermined function and $J(x,t)\in C^{2,1}(B(x_0,2R)\times[0,\infty))$ is the unique solution of \eqref{a1} on $B(x_0,2R)\times[0,\infty)$.
		
		For any fixed $T>0$, we consider the maximal value of $G$ on $B(x_0,2R)\times[0,T]$, and denote the maximal value point $(\bx,\bt)$. Without of loss generality, we can assume $G(\bx,\bt)>0$, and hence $\bt>0$. At this point, we have
		\begin{eqnarray}\label{bb}
			0\ge\l G(\bx,\bt)&=&(\bt\d \Phi-\Phi)JF(\bx,\bt)+\bt\Phi\l(JF)+2\left\langle\nabla(tJF),\nabla\Phi\right\rangle(\bx,\bt)\nonumber\\
			&=&\left(\frac{\d\Phi}{\Phi}-\frac{1}{\bt}\right)G(\bx,\bt)+\bt\Phi \l(JF)-2\frac{|\nabla \Phi|^2}{\Phi^2}G(\bx,\bt).
		\end{eqnarray}
		By Lemma \ref{C3}, \eqref{bb} and Lemma \ref{301} with $\b=\b(n)$ and $\delta=\delta(n)$, we obtain
		\begin{eqnarray}
			0&\ge&\left(\frac{\d\Phi}{\Phi}-\frac{1}{\bt}-2\frac{|\nabla \Phi|^2}{\Phi^2}\right)G(\bx,\bt)+\bt\Phi \left(2JF^2+2\left(\frac{1}{\b}-1\right)\left\langle\nabla (JF),\nabla\ln w\right\rangle\right)(\bx,\bt)\nonumber\\
			&=&\left(\frac{\d\Phi}{\Phi}-\frac{1}{\bt}-2\frac{|\nabla \Phi|^2}{\Phi^2}\right)G(\bx,\bt)+2\bt\Phi JF^2- 2\left(\frac{1}{\b}-1\right)\left\langle\nabla \Phi,\nabla\ln w\right\rangle\frac{G}{\Phi}(\bx,\bt),\nonumber
		\end{eqnarray}
		which is equivalent to 
		\begin{equation}\label{305}
			0\ge\left(\bt\d \Phi-\Phi-2\bt\frac{|\nabla\Phi|^2}{\Phi}\right)JG(\bx,\bt)+2G^2(\bx,\bt)- 2\left(\frac{1}{\b}-1\right)\left\langle\nabla \Phi,\nabla\ln w\right\rangle \bt JG(\bx,\bt).
		\end{equation}
		Cauchy-Schwarz inequality and basic inequality  give
		\begin{eqnarray}\label{cs}
			- 2\left(\frac{1}{\b}-1\right)\left\langle\nabla \Phi,\nabla\ln w\right\rangle \bt JG(\bx,\bt)&\ge&-2\left|\frac{1}{\b}-1\right||\nabla\Phi|F^{\frac{1}{2}}\cdot \bt JG(\bx,\bt)\nonumber\\
			&\ge&-G^2(\bx,\bt)-\left(\frac{1}{\b}-1\right)^2\frac{|\nabla\Phi|^2}{\Phi}\cdot \bt JG(\bx,\bt).
		\end{eqnarray}
		Substituting \eqref{cs} into \eqref{305} yields
		\begin{equation}\label{306}
			G(\bx,\bt)\le\left(\bt|\d\Phi|+\Phi+\left(2+\left(\frac{1}{\b}-1\right)^2\right)\bt\frac{|\nabla\Phi|^2}{\Phi}\right)J(\bx,\bt),
		\end{equation}
		then we choose $\Phi=\phi^2$, where $\phi\in C^{\infty}(B(x_0,2R))$ is same as in Lemma \ref{ecf}, and we have
		\begin{equation}\label{307}
			G(\bx,\bt)\le C(n,p)\left(1+\frac{T}{R^2}\right),
		\end{equation}
		which is due to $\bt\le T$, $\phi(\bx)\le 1$ and $J(\bx,\bt)\le 1$. Hence
		\begin{equation}\label{308}
			TJ_{2R}(T)\sup_{B(x_0,R)}F(x)\le\sup_{B(x_0,R)}TJ(x,T)F(x)\le G(\bx,\bt)\le C(n,p)\left(1+\frac{T}{R^2}\right),
		\end{equation}
		where $J_{2R}(T)$ is defined by \eqref{jr}. Because $T$ is arbitrary, we set $T=R^2$ in \eqref{308} and obtain
		\begin{equation}
			\sup_{B(x_0,R)}F\le Ce^{Ck(1+Ck^{\frac{n}{2p-n}})}\cdot\frac{1}{R^2}=\frac{C}{R^2},
		\end{equation}
		where $C=C(n,p)$.

	\end{proof}

	\section{\textbf{Generalization of nonlinear item}}\label{S4}
	In this section, we generalize previous result for equation \eqref{AC} to the following nonlinear equation
	\begin{equation}\label{LEC}
		\Delta u+au^{s}-bu^t=0,
	\end{equation}
	where $a\ge0,b\ge0$ and $s\neq t$.
	
	Via same transformation as \eqref{tran} and auxiliary function $F$ as \eqref{F}, similar computation as in the proof of Proposition \ref{k1} gives the following proposition. 
	
	\begin{prop}\label{k}
		Let $u$ be a positive smooth solution of \eqref{LEC} on a domain $\O\subset\M$ and $F$ be defined by \eqref{F}. Then we have
		\begin{eqnarray}\label{4k}
			\Delta F&=&2w^{-2}\left|\nabla^2w-\frac{\Delta w}{n} g\right|^2+2w^{-2}Ric(\nabla w,\nabla w)+2\left(\frac{1}{\b}-1\right)\left\langle\nabla F,\nabla \ln w\right\rangle\nonumber\\
			&&+\left(au^{s-1}\left(\frac{4}{n}(1+\b)+2(1-s)\right)-bu^{t-1}\left(\frac{4}{n}(1+\b)+2(1-t)\right)\right)\frac{|\nabla w|^2}{w^2}\nonumber\\
			&&+\left(\frac{2}{n}\left(1+\frac{1}{\b}\right)^2-2\right)\frac{|\nabla w|^4}{w^4}+\frac{2\b^2}{n}\left(au^{s-1}-bu^{t-1}\right)^2.
		\end{eqnarray}	
	\end{prop}
	Before futher statement, for any real number $n\in[1,\infty)$, we define function
	\begin{eqnarray}
		p(n)=\begin{cases}
			\infty    \quad\qquad\qquad\text{if}\quad n=1\nonumber\\
			\frac{n+3}{n-1}\qquad\qquad\,\,\text{if}\quad n\in(1,\infty).\nonumber
		\end{cases}
	\end{eqnarray}
	
	By Proposition \ref{k}, we immediately get the following key lemma.
	
	\begin{lem}\label{kl}
		Let $u$ be a positive smooth solution of \eqref{LEC} on a domain $\O\subset\M$ and $F$ be defined by \eqref{F}. For the following cases:\\
		{\rm(I)} $s\in (-\infty,1]$ and $t>1$,\\
		{\rm(II)} $s\in(1,p(n))$ and $t>s$,\\
		there exists $\b=\b(n,s,t)>0$ such that 
		\begin{equation}\label{43}
			\Delta F\ge 2w^{-2}Ric(\nabla w,\nabla w)+2\left(\frac{1}{\b}-1\right)\left\langle\nabla F,\nabla \ln w\right\rangle+LF^2\quad\text{on\,\, $\O$,}
		\end{equation}
		where $L=L(n,s,t)>0$.
	\end{lem}
	\begin{proof}[\textbf{Proof}]
		Case 1:	$s\le 1$ and $t>1$. Then for any $\b\in(0,\frac{n}{2}(t-1)]$, we have
		\begin{eqnarray}\label{44}
			au^{s-1}\left(\frac{4}{n}(1+\b)+2(1-s)\right)-bu^{t-1}\left(\frac{4}{n}(1+\b)+2(1-t)\right)\ge\frac{4}{n}\left(au^{s-1}-bu^{t-1}\right).
		\end{eqnarray}	
		Notice that when $\b>0$ is sufficiently small, by basic inequality, one can see	
		\begin{eqnarray}\label{45}
			&&	\left(\frac{2}{n}\left(1+\frac{1}{\b}\right)^2-4\right)\frac{|\nabla w|^4}{w^4}+\frac{2\b^2}{n}\left(au^{s-1}-bu^{t-1}\right)^2\nonumber\\
			&&\ge\frac{4}{n}\sqrt{(1+\b)^2-2n\b^2}\cdot|au^{s-1}-bu^{t-1}|\cdot\frac{|\nabla w|^2}{w^2}\nonumber\\
			&&\ge\frac{4}{n}|au^{s-1}-bu^{t-1}|\cdot\frac{|\nabla w|^2}{w^2}.
		\end{eqnarray}	
		Substituting \eqref{44} and \eqref{45}	into \eqref{4k} gives \eqref{43}.
		
		Case 2:	$s\in (1,1+\frac{2}{n}]$ and $t>s$.
		Then for any $\b\in(0,\frac{n}{2}(s-1)]$, we have
		\begin{eqnarray}\label{46}
			&&au^{s-1}\left(\frac{4}{n}(1+\b)+2(1-s)\right)-bu^{t-1}\left(\frac{4}{n}(1+\b)+2(1-t)\right)\nonumber\\
			&&\ge \left(\frac{4}{n}(1+\b)+2(1-s)\right)\left(au^{s-1}-bu^{t-1}\right)\ge-\frac{4}{n}|au^{s-1}-bu^{t-1}|.
		\end{eqnarray}		
		Then same reason as Case 1 gives \eqref{43}.	
		
		Case 3:	$s\in (1+\frac{2}{n},1+\frac{4}{n}]$ and $t>s$.
		Then for any $\b\in(0,\frac{n}{2}(s-1)-1)$, we have
		\begin{eqnarray}\label{47}
			&&au^{s-1}\left(\frac{4}{n}(1+\b)+2(1-s)\right)-bu^{t-1}\left(\frac{4}{n}(1+\b)+2(1-t)\right)\nonumber\\
			&&\ge \left(\frac{4}{n}(1+\b)+2(1-s)\right)\left(au^{s-1}-bu^{t-1}\right)\ge-\frac{4}{n}|au^{s-1}-bu^{t-1}|.
		\end{eqnarray}		
		Then same reason as Case 1 gives \eqref{43}.	
		
		Case 4:	$s\in (1+\frac{4}{n},p(n))$ and $t>s$. For any $l\in(2,\infty)$ and $\b\in(0,\frac{4}{nl-2}]$, we have
		\begin{eqnarray}\label{48}
			&&	\left(\frac{2}{n}\left(1+\frac{1}{\b}\right)^2-l\right)\frac{|\nabla w|^4}{w^4}+\frac{2\b^2}{n}\left(au^{s-1}-bu^{t-1}\right)^2\nonumber\\
			&&\ge\frac{4}{n}\sqrt{(1+\b)^2-\frac{nl}{2}\b^2}\cdot|au^{s-1}-bu^{t-1}|\cdot\frac{|\nabla w|^2}{w^2}.
		\end{eqnarray}		
		Substituting \eqref{48} into \eqref{4k} yields	
		\begin{eqnarray}\label{2bb}
			\d F&\ge& 2w^{-2}Ric(\nabla w,\nabla w)+2\left(\frac{1}{\b}-1\right)\left\langle\nabla F,\nabla\ln w\right\rangle+(l-2)F^2	\nonumber\\
			&&+\left(g(\b,n,l)(au^{s-1}-bu^{t-1})-2(sau^{s-1}-tbu^{t-1})\right)\frac{|\nabla w|^2}{w^2},
		\end{eqnarray}
		where	
		\begin{equation}
			g(\b,n,l)=\frac{4}{n}(1+\b)+\frac{4}{n}\sqrt{(1+\b)^2-\frac{nl}{2}\b^2}+2.\nonumber
		\end{equation}
		Now, we choose $l>2$ such that $s-1=\frac{4l}{nl-2}$ and set $\b=\frac{4}{nl-2}$. Then \eqref{2bb} gives \eqref{43} with $L=l-2$.	
		
	\end{proof}
	\begin{rmk}
		\rm	The careful reader find Lemma \ref{kl} can be improved because we do not seriously deal with the cofficient of $\frac{|\nabla w|^2}{w^2}$. In many specific cases,  this cofficient can be positive and can yield Liouville theorem of \eqref{LEC} on negative curved space as Theorem \ref{main2}. Anyway, our attention is Liouville theorem of \eqref{LEC} on manifolds with  non-negative Ricci curvature, hence Lemma \ref{kl} is enough for it. Interested readers may try to improve it and following Theorem \ref{4main2}.
	\end{rmk}

	By nearly same argument as in Section \ref{S2} and Lemma \ref{kl}, we have the following theorem. 
	\begin{thm}\label{4main}
		Let $(\M,g)$ be an n-dimensional complete Riemannian manifold with $Ric\ge-Kg$, where $K\ge 0$. If $u$ is a positive solution of \eqref{LEC} on $B(x_0,2R)$, for the following cases:\\
		{\rm(I)} $s\in(-\infty,1]$ and $t>1$,\\
		{\rm(II)} $s\in(1,p(n))$ and $t>s$,\\
		we have
		\begin{equation}\label{strong}
			\sup_{B(x_0,R)}\frac{|\nabla u|^2}{u^2}\le C\left(K+\frac{1}{R^2}\right),
		\end{equation}	
		where $C=C(n,s,t)$.
		
	\end{thm}
	As its natural corollary, we can see 
	
	\begin{cor}\label{4mainc}
		Let $(\M,g)$ be an n-dimensional complete Riemannian manifold with $Ric\ge-Kg$, where $K\ge 0$. If $u$ is a positive solution of \eqref{LEC} on $B(x_0,2R)$, for the following cases:\\
		{\rm(I)} $s\in(-\infty,1]$ and $t>1$,\\
		{\rm(II)} $s\in(1,p(n))$ and $t>s$,\\
		we have
		\begin{equation}\label{strong}
			\sup_{B(x_0,R)}u\le e^{C(\sqrt{K}R+1)}\inf_{B(x_0,R)}u,
		\end{equation}	
		where $C=C(n,s,t)$.
		
	\end{cor}

	\begin{thm}\label{4main2}
		Let $(\M,g)$ be an n-dimensional complete Riemannian manifold with non-negative Ricci curvature. If $u$ is a positive solution of \eqref{LEC} on $\M$ with $a>0,b>0$, for the following two cases:\\
		{\rm(I)} $s\in(-\infty,1]$ and $t>1$,\\
		{\rm(II)} $s\in(1,p(n))$ and $t>s$,\\
		then $u\equiv\left(\frac{a}{b}\right)^{\frac{1}{t-s}}$.
		
	\end{thm}
	\begin{rmk}
		\rm(A) Except Allen-Cahn equation, Theorem \ref{4main2} also can be applied for static Fisher-KPP equation (with $a=b>0, s=1, t=2$ in equation \eqref{LEC}) and static Newell-Whitehead equation (with $a>0,b>0, s=1, t=3$ in equation \eqref{LEC}). Therefore, \emph{the static positive solution of Fisher-KPP equation (or Newell-Whitehead equation) on the complete Riemannian manifold with nonnegative Ricci curvature must be constant $1$}.\\	(B) As corollary of Theorem \ref{4main} and strong maximal principle, we can obtain the Liouville theorem of equation \eqref{LEC} with $a>0$ and $b=0$, that is, \emph{for $s\in(-\infty,p(n))$, the nonnegative solution of \eqref{LEC} on the complete Riemannain manifold with nonnegative Ricci curvature must be trivial.} However, this result is not satisfied because the critical number $s$ of existence and non-existence of positive solutions of  equation \eqref{LEC} ($a>0$ and $b=0$) should be $\frac{n+2}{n-2}$ for $n>2$ and $\infty$ for $n\in[1,2]$ by Gidas-Spruck \cite{GS}. The present method and tricks in the paper can not achieve this result and we have verified it (and more general case) in the paper \cite{LU2}. \\
		(C) At $a>0$ case (in equation \eqref{LEC}), We mention that the cases $b>0$ and $b=0$ are very different about the behavior of positive  solutions of \eqref{LEC}. If $b=0$, one can get sharp pointwise uniform estimate about solutions on any domains on manifolds (even metric spaces), however, for $b>0$ case, it seems difficult to derive such estimate without  prior upper bound condition.\\
		(D) As same as Allen-Cahn equation, one also can get Liouville theorem of \eqref{LEC} on closed manifolds under a loose Ricci curvature condition. On negative curved complete manifolds case, one also can obtain Liouville theorem of \eqref{LEC} as Theorem \ref{main2}. For example, by imitating proof of Theorem \ref{main2}, one can directly show that \emph{the positive classical solutions of static Whitehead-Newell equation on complete Riemannian manifolds with $Ric\ge-Kg$ and $K\le aN(n)$ must be constant, where $N(n)$ is given by \eqref{N}.} One can try to give similar result for Fisher-KPP equation by Proposition \ref{k} and argument in the proof of Theorem \ref{main2}. 
	\end{rmk}

	Now, combining same argument as Section 3 and Lemma \ref{kl}, we have the following theorem.
	
	\begin{thm}\label{4001}
		Let $(\M,g)$ be an n-dimensional complete Riemannian manifold. If $u$ is a positive solution of \eqref{LEC} on $B(x_0,2R)$, then for the following two cases:\\
		{\rm(I)} $s\in(-\infty,1]$ and $t>1$,\\
		{\rm(II)} $s\in(1,p(n))$ and $t>s$,\\
		and for 
		any $p>\frac{n}{2}$, there exists constants $k=k(n,p)>0$ and $C=C(n,p,s,t)>0$ such that if $k(p,R)\le k$, we have
		\begin{equation}
			\sup_{B(x_0,R)}\frac{|\nabla u|^2}{u^2}\le \frac{C}{R^2}.
		\end{equation}	
		
	\end{thm}
	
	\begin{cor}\label{4main1cc}
		Let $(\M,g)$ be an n-dimensional complete Riemannian manifold. If $u$ is a positive solution of \eqref{LEC} on $B(x_0,2R)$, then for the following two cases:\\
		{\rm(I)} $s\in(-\infty,1]$ and $t>1$,\\
		{\rm(II)} $s\in(1,p(n))$ and $t>s$,\\
		and for 
		any $p>\frac{n}{2}$, there exist constants $k=k(n,p)>0$ and $C=C(n,p,s,t)>0$ such that if $k(p,R)\le k$, 	we have
		\begin{equation}\label{strong}
			\sup_{B(x_0,R)}u\le C\inf_{B(x_0,R)}u.
		\end{equation}	
		
	\end{cor}

	Now, we provide a question that is inspired by Theorem \ref{4main2}.
	Define

	\begin{equation}
		E=\left\{\begin{array}{l|l}
			(s,t) \in \mathbb{R}^2 & \begin{array}{l}
				\text{The postive classical solutions of \eqref{LEC} on complete } \\
				\text{Riemannian manifolds with non-negative Ricci }\\
				\text{curvature must be constant	for any $a>0$ and $b>0$.}
			\end{array}
		\end{array}\right\} .
	\end{equation}
	A natural question is how to determine set $E$.

	\section{\textbf{Estimate under Bakry-\'{E}mery curvature}}\label{S5}
	In this section, we consider a slight generalization of equation \eqref{LEC}  as follows.
	\begin{equation}\label{V}
		\d_V u+au^s-bu^t=0,
	\end{equation}
	where $\d_V u=\d u+\left\langle\nabla u,V\right\rangle$, $V$ is a $C^1$ vector field, $a\ge 0,b\ge0$ and $s\neq t$. When $V$ vanishes, equation \eqref{V} becomes \eqref{LEC}.
	
	We also want to get similar estimate as in Section \ref{S4} for equation \eqref{V}. For this, we need to recall previous proof and analyse present suitable condition. As before, we define auxiliary function $F$ as \eqref{tran} and \eqref{F}, by similar computation, then we have
	\begin{eqnarray}\label{5k}
		\Delta_V F&=&2w^{-2}\left|\nabla^2w\right|^2+2w^{-2}Ric_V(\nabla w,\nabla w)+2\left(\frac{1}{\b}-1\right)\left\langle\nabla F,\nabla\ln w\right\rangle\nonumber\\
		&&+2\left(au^{s-1}(1-s)-bu^{t-1}(1-t)\right)\frac{|\nabla w|^2}{w^2}-2\frac{|\nabla w|^4}{w^4},
	\end{eqnarray}	
	where
	\begin{equation}\label{RV}
		Ric_V:=Ric-\frac{1}{2}L_V g,
	\end{equation}
	$ L_V g$ is the Lie derivative of Riemannian metric $g$ on the vector field $V$. Here we use the following Bochner formula to \eqref{5k}\footnote{One can obtain it by using classical Bochner formula and definition \eqref{RV}.}.
	\begin{equation}
		\d_V |\nabla w|^2=2|\nabla^2 w|^2+2Ric_V(\nabla w,\nabla w)+2\left\langle\nabla\d_V w,\nabla w\right\rangle.
	\end{equation}
	
	Notice that for any $m>n$, we have
	\begin{eqnarray}\label{kk}
		|\nabla^2 w|^2&=&\left|\nabla^2 w-\frac{\d w}{n}g\right|^2+\frac{1}{n}\left(\d_V w-\left\langle\nabla w,V\right\rangle\right)^2\nonumber\\
		&\ge&\left|\nabla^2 w-\frac{\d w}{n}g\right|^2+\frac{1}{n}\left(\frac{(\d_V w)^2}{\frac{m}{n}}-\frac{\left\langle\nabla w,V\right\rangle^2}{\frac{m}{n}-1}\right)\nonumber\\
		&\ge&\frac{1}{m}(\d_V w)^2-\frac{1}{m-n}\left\langle\nabla w,V\right\rangle^2,
	\end{eqnarray}
	where we use basic inequality $(a-b)^2\ge\frac{a^2}{t}-\frac{b^2}{t-1}$ for any $t>1$.
	
	Substituting \eqref{kk} into \eqref{5k} and using equation \eqref{V} yield
	
	\begin{eqnarray}\label{5kk}
		\Delta_{V} F&\ge&2w^{-2}Ric_V^m(\nabla w,\nabla w)+2\left(\frac{1}{\b}-1\right)\left\langle\nabla F,\nabla \ln w\right\rangle\nonumber\\
		&&+\left(au^{s-1}\left(\frac{4}{m}(1+\b)+2(1-s)\right)-bu^{t-1}\left(\frac{4}{m}(1+\b)+2(1-t)\right)\right)\frac{|\nabla w|^2}{w^2}\nonumber\\
		&&+\left(\frac{2}{m}\left(1+\frac{1}{\b}\right)^2-2\right)\frac{|\nabla w|^4}{w^4}+\frac{2\b^2}{m}\left(au^{s-1}-bu^{t-1}\right)^2,
	\end{eqnarray}	
	where
	\begin{equation}
		Ric_V^m:=Ric_V-\frac{1}{m-n}V^{\flat}\otimes V^{\flat},
	\end{equation}
	is called $m$-dimensional Bakry-\'{E}mery Ricci curvature, $V^{\flat}$ is the dual $1$-form of vector field $V$ and $m>n$.
	
	By same argument as in the proof of Lemma \ref{kl}, \eqref{5kk} can give
	\begin{lem}\label{5kl}
		Let $u$ be a positive smooth solution of \eqref{V} on a domain $\O\subset\M$, $F$ be defined by \eqref{F} and $m>n$. For the following two cases:\\
		{\rm(I)} $s\in(-\infty,1]$ and $t>1$,\\
		{\rm(II)} $s\in(1,p(m))$ and $t>s$,\\
		there exists $\b=\b(m,s,t)>0$ such that 
		\begin{equation}\label{53}
			\Delta_{V} F\ge 2w^{-2}Ric_V^m(\nabla w,\nabla w)+2\left(\frac{1}{\b}-1\right)\left\langle\nabla F,\nabla \ln w\right\rangle+LF^2\qquad\text{on\quad $\O$},
		\end{equation}
		where $L=L(m,s,t)>0$.
	\end{lem}
	
	Via Bakry-Qian's Laplacian comparison theorem \cite[Theorem 4.2]{BQ}, we also have the following  cut-off function 
	as in the proof of Theorem \ref{main1}.  
	
	\begin{lem}\label{VCF}
		Let $(\M,g)$ be an n-dimensional complete Riemannian manifold with $Ric_V^m\ge -Kg$, where $V$ is a smooth vector field, $m>n$ and $K\ge 0$. Then we have cut-off function $\Phi\in C^{\infty}_{0}(B(x_0,2R))$ such that\\
		{\rm (i)} $\Phi(x)=\phi(d(x_0,x))$, where $\phi$ is a non-increasing function on $[0,\infty)$ and
		\begin{eqnarray}
			\Phi(x)=
			\begin{cases}
				1\qquad\qquad\qquad\,\text{if}\qquad x\in B(x_0,R)\nonumber\\
				0\qquad\qquad\qquad\,\text{if}\qquad x\in B(x_0,2R)\setminus B(x_0,\frac{3}{2}R).
			\end{cases}
		\end{eqnarray}
		{\rm (ii)} For any fixed $l\in(0,1)$,
		\begin{equation}
			\frac{|\nabla \Phi|}{\Phi^{l}}\le\frac{C(l)}{R}.\nonumber
		\end{equation}
		{\rm (iii)}
		\begin{equation}
			\Delta_V \Phi\ge-\frac{C\sqrt{K}}{R}-\frac{C}{R^2}\nonumber
		\end{equation}
		on $B(x_0,2R)$ and outside of cut locus of $x_0$, where $C=C(m)$.	
	\end{lem}
	
	Combining Lemma \ref{5kl} and Lemma \ref{VCF}, one can obtain the following result by imitating same process as in the proof of Theorem \ref{main1} (or Theorem \ref{4main}). 
	
	\begin{thm}\label{5main}
		Let $(\M,g)$ be an n-dimensional complete Riemannian manifold with $Ric_V^m\ge-Kg$, where $V$ is a smooth vector field, $K\ge 0$ and $m>n$. If $u$ is a positive solution of \eqref{V} on $B(x_0,2R)$, for the following two cases:\\
		{\rm(I)} $s\in(-\infty,1]$ and $t>1$,\\
		{\rm(II)} $s\in(1,p(m))$ and $t>s$,\\
		we have
		\begin{equation}\label{strong5}
			\sup_{B(x_0,R)}\frac{|\nabla u|^2}{u^2}\le C\left(K+\frac{1}{R^2}\right),
		\end{equation}	
		where $C=C(m,s,t)$.
		
	\end{thm}
	
	We omit Harnack inequality of corollary of Theorem \ref{5main} and give  Liouville theorem as follows.
	
	\begin{thm}\label{5main2}
		Let $(\M,g)$ be an n-dimensional complete Riemannian manifold with $Ric_V^m\ge 0$, where $V$ is a smooth vector field and $m>n$. If $u$ is a positive solution of \eqref{V} on $\M$ with $a>0,b>0$, for the following two cases:\\
		{\rm(I)} $s\in(-\infty,1]$ and $t>1$,\\
		{\rm(II)} $s\in(1,p(m))$ and $t>s$,\\then $u\equiv\left(\frac{a}{b}\right)^{\frac{1}{t-s}}$.
		
	\end{thm}
	
	For Allen-Cahn equation or static Fisher-KPP equation with gradient item $\left\langle\nabla u,V\right\rangle$ for some smooth vector field $V$,  if we find $m>n$ and $Ric_V^m\ge 0$, then we deduce that the positive solution $u\equiv 1$ by Theorem \ref{5main2}. Actually, as Theorem \ref{main2}, one also can loose Bakry-\'{E}mery curvature to a little negative case.

	
	

	\section*{\textbf{Acknowledgements}}The  author would like to thank Professor Jiayu Li for his support and  encouragement. After the completion of this work (January 2023), we find that Proposition \ref{k1} can be generalized for more nonlinearities and the logarithmic gradient estimate also can be generalized, see \cite{LU2}. Recently, a parabolic generalization of present work was collected in \cite{LU}.


\begin{thebibliography}{99}
		\bibitem{BQ}
		Bakry Dominique and Qian Zhongmin, \emph{Volume comparison theorems without Jacobi fields.} Current trends in potential theory, 115–122, Theta Ser. Adv. Math., \textbf{4}, Theta, Bucharest, 2005.
		\bibitem{DWZ}
		Dai Xianzhe, Wei Guofang and Zhang Zhenlei, \emph{Local Sobolev constant estimate for integral Ricci curvature bounds}. Adv. Math. \textbf{325} (2018), 1–33. 
		
		\bibitem{GS}
		Gidas Basilis and Spruck Joel, \emph{Global and local behavior of positive solutions of nonlinear elliptic equations}. Comm. Pure Appl. Math. \textbf{34} (1981), no. 4, 525–598.
		\bibitem{H}
		Hou Songbo, \emph{Gradient estimates for the Allen-Cahn equation on Riemannian manifolds.} Proc. Amer. Math. Soc. \textbf{147} (2019), no. 2, 619–628.
		\bibitem{LY} Li Peter and Yau Shing-Tung, \emph{On the parabolic kernel of the Schrödinger operator}. Acta Math. \textbf{156} (1986), no. 3-4, 153–201.
		\bibitem{LU0}
		Lu Zhihao, \emph{Differential Harnack inequalities for semilinear parabolic equations on Riemannian manifolds II: Integral curvature condition.} Nonlinear Anal. 239 (2024), Paper No. 113426, 28 pp.
		\bibitem{LU} 
		Lu Zhihao, \emph{Differential Harnack inequalities for Fisher-KPP type equations on Riemannian manifolds.} Preprint, \href{https://arxiv.org/abs/2404.06755}{arXiv:2404.06755}.
		\bibitem{LU2}
		Lu Zhihao, \emph{Logarithmic gradient estimate and universal bounds for semilinear elliptic equations revisited.}	 Preprint, \href{https://arxiv.org/abs/2308.14026}{arXiv:2308.14026}.
		
		
		\bibitem{ZZ1} Zhang Qi S. and Zhu Meng, \emph{Li-Yau gradient bounds on compact manifolds under nearly optimal curvature conditions}. J. Funct. Anal. \textbf{275} (2018), no. 2, 478–515.
		
		\bibitem{ZZ2} Zhang Qi S. and Zhu Meng, \emph{Li-Yau gradient bound for collapsing manifolds under integral curvature condition}. Proc. Amer. Math. Soc. \textbf{145} (2017), no. 7, 3117–3126.
		
	\end{thebibliography}
\end{document}